\documentclass{amsart}
\title{On arithmetic sums involving divisor functions in two variables}
\usepackage{amssymb,amsmath,amsthm,epsfig,graphics,latexsym}
 \theoremstyle{definition}
 \newtheorem{definition}{Definition}
  \theoremstyle{plain}

  \newtheorem{theorem}   {Theorem}
  \newtheorem*{maintheorem} {Main Theorem}

  \theoremstyle{remark}
  \newtheorem*{rmk}{{\bf Remark}}
  \newtheorem*{problem} {\bf Problem}
\begin{document}
  \author{Adama Diene and Mohamed El Bachraoui}
  \address{Dept. Math. Sci,
 United Arab Emirates University, PO Box 17551, Al-Ain, UAE}
 \email{adiene@uaeu.ac.ae, melbachraoui@uaeu.ac.ae}
  \begin{abstract}
  We prove the analogue of an identity of Huard, Ou, Spearman and Williams and apply it
  to evaluate a variety of sums involving divisor functions in two variables.
  It turns out that these sums count representations of positive integers involving radicals.
  \end{abstract}
  \date{\textit{\today}}
  \maketitle
 \section{Introduction}
 Throughout, let $\mathbb{N}$, $\mathbb{N}_0=\mathbb{N}\cup\{0\}$, $\mathbb{Z}$, $\mathbb{Q}$, and $\mathbb{C}$ 
 be the sets of positive integers, nonnegative integers,
 integers, rational numbers, and complex numbers respectively. Let $\sigma_k (n)$ be the sum of $k$th powers of the divisors of $n$
 with the assumption that $\sigma_k (n)=0$ if $n \not\in \mathbb{N}$. Convolution sums of the form
 \[ \sum_{m=1}^{n-1} \sigma_{r}(m) \sigma_{s}(n-m), \quad (r,s \in \mathbb{N}) \]
 have been investigated by many mathematicians including in chronicle order Glaisher ~\cite{Glaisher1, Glaisher2, Glaisher3},
 Ramanujan ~\cite{Ramanujan1}, MacMahon ~\cite{MacMahon}, Lahiri ~\cite{Lahiri1, Lahiri2}, and Melfi ~\cite{Melfi}.
 A variety of such convolution sums have been evaluated using advanced techniques such as the theory of modular forms and the theory of $q$-series.
  In reference ~\cite{HOSW}, Huard, Ou, Spearman, and Williams
 among other things gave elementary proofs for many of these convolution sums along with many new identities.
 See also Williams ~\cite[Chapter 13]{Williams}.
 The authors' main argument is the following theorem.
\begin{theorem} \label{hosw-1}
 Let $2 \leq n\in\mathbb{N}$ and let
 \[
 \mathcal{B}(n) = \{(a,b,x,y)\in\mathbb{N}^4:\ ax+by=n \}.
 \]
 Let $f:\mathbb{Z}^4 \to \mathbb{C}$ be such that
 \[
 f(a,b,x,y)-f(x,y,a,b)= f(-a,-b,x,y)- f(x,y,-a,-b)
 \]
 for all $a,b,x,y \in\mathbb{Z}$. Then
 \[
 \sum_{(a,b,x,y)\in \mathcal{B}(n)} \bigl( f(a,b,x,-y)-f(a,-b,x,y)+f(a,a-b,x+y,y)
 \]
 \[
 -f(a,a+b,y-x,y)+f(b-a,b,x,x+y)-f(a+b,b,x,x-y) \bigr)
 \]
 \[
 =
 \sum_{d\in\mathbb{N}, d\mid n}\sum_{x\in\mathbb{N}, x<d}
 \bigl( f(0,n/d,x,d)+f(n/d,0,d,x)+f(n/d,n/d,d-x,-x)
 \]
 \[
 -f(x,x-d,n/d,n/d)-f(x,d,0,n/d)-f(d,x,n/d,0) \bigr).
 \]
 \end{theorem}
 For instance, by application of Theorem ~\ref{hosw-1} to the function $f(a,b,x,y)=xy$ the authors
 obtained the Besge's identity
 \[
 \sum_{m=1}^{n-1} \sigma(m) \sigma(n-m) = \frac{1}{12} (5 \sigma_3(n) + (1-6n)\sigma(n))
 \]
 and by application of the same theorem to the function $f(a,b,x,y)=x^3 y + x y^3$  the authors reproduced the Glaisher's formula
 \[
 \sum_{m=1}^{n-1} \sigma(m) \sigma_{3}(n-m) = \frac{1}{240} (21 \sigma_5 (n) +(10-30n)\sigma_3(n) -\sigma(n)).
 \]
 They also used Theorem ~\ref{hosw-1} to deduce formulas for the sums
 \[ \sum_{m=1}^{n-1} \sigma_{r}(m) \sigma_{s}(n-m) \]
 for odd $r$ and $s$ such that $r+s \in \{2,4,6,8,12\}$.
 However, as was pointed out by the authors formulas for such sums for odd $r$ and $s$ such that $r+s =10$
 remain out of reach of their methods. These had been evaluated earlier by Lahiri ~\cite{Lahiri2}
 using Ramanujan's tau function $\tau(n)$. Quite recently, Royer ~\cite{Royer} used quasimodular forms to reproduce many
 convolution sums involving the divisor functions.

 It is natural to ask what would happen if in Theorem ~\ref{hosw-1} for instance instead of the set $\mathcal{B}(n)$ one takes the sum over the set
 \[
  \mathcal{B}'(n) = \{(a,b,x,y)\in\mathbb{N}^4:\ ax+by=n,\ \text{and\ } \gcd(a,b)= \gcd(x,y)=1 \}.
 \]
 The main theorem of our paper deals with such an analogue of Theorem ~\ref{hosw-1}.
 \begin{maintheorem} \label{main1}
 Let $2\leq n\in\mathbb{N}$ and
 let $f:\mathbb{Z}^4 \to \mathbb{C}$ be such that
 \[
 f(a,b,x,y)-f(x,y,a,b)= f(-a,-b,x,y)- f(x,y,-a,-b)
 \]
 for all $a,b,x,y \in\mathbb{Z}$. Then
 \begin{equation}\label{main-eq}
 \sum_{(a,b,x,y)\in \mathcal{B}'(n)} \bigl( f(a,b,x,-y)-f(a,-b,x,y)+f(a,a-b,x+y,y)
 \end{equation}
 \[
 -f(a,a+b,y-x,y)+f(b-a,b,x,x+y)-f(a+b,b,x,x-y) \bigr)
 \]
 \[
 =
 \sum_{\substack{1\leq t< n\\ (t,n)=1}}
 \bigl( f(1,0,n,t)-f(n,t,1,0)+f(0,1,t,n)
 \]
 \[
 -f(t,n,0,1)+f(1,1,n-t,-t)-f(n-t,-t,1,1) \bigr).
\]
 \end{maintheorem}
 Our work suggests the following definition of the sum of divisors function in two variables.
 \begin{definition} \label{prime-divisors}
 For $m,n \in \mathbb{N}$ and $r, s \in\mathbb{N}_0$ let
 \[
 \sigma'_{r,s}(m,n)
 =
  \sigma'_{s,r}(n,m)
 =
 \sum_{\substack{(d,e)\in\mathbb{N}^2\\ d\mid m,\ e\mid n\\ (d,e)=(\frac{m}{d},\frac{n}{e})=1}} d^r e^s.
 \]
 \end{definition}
 It is easily seen that
 \begin{equation}\label{pre-identity}
 \sum_{m=1}^{n-1} \sigma'_{r,s}(m,n-m) = \sum_{m=1}^{n-1}\sigma'_{s,r}(m,n-m)
 \end{equation}
 \[
 =\sum_{(a,b,x,y)\in\mathcal{B}'(n)}x^r y^s =\sum_{(a,b,x,y)\in\mathcal{B}'(n)}x^s y^r
 = \sum_{(a,b,x,y)\in\mathcal{B}'(n)}a^r b^s = \sum_{(a,b,x,y)\in\mathcal{B}'(n)}a^s b^r.
 \]
 We shall use Main Theorem and the relation (\ref{pre-identity}) to evaluate the sums
 \[
 \sum_{m=1}^{n-1}\sigma'_{r,s}(m,n-m)
 \]
 for $r\equiv s\equiv 1\pmod 2$ such that $r+s \in \{2,4,6,8,12\}$.
 Whereas the sums $\sum_{m=1}^{n-1} \sigma_r (m) \sigma_s (n-m)$ for these values of $r$ and $s$ are in terms of the divisor functions
 $\sigma_k (n)$, their analogues $\sum_{m=1}^{n-1}\sigma'_{r,s}(m,n-m)$ are in terms of the function
 $\psi_{s}(n)$ defined by
 \[
 \psi_s(n) = \sum_{d\mid n} \mu(d) d^s\quad \text{for $n\in\mathbb{N}$ and $s\in\mathbb{Z}\setminus\{0\}$},
 \]
 where $\mu(n)$ denotes the M\"{o}bius mu function.
 To be fair to Huard, Ou, Spearman, and Williams we note that our proofs are essentially the same as theirs for the corresponding
 results in reference~ \cite{HOSW}.
 As a new development, we will show that our identities actually count representations of positive integers.
 \begin{rmk}
 While formulas are known for the sums
  \(
  \sum_{m=1}^{n-1}\sigma_r(m) \sigma_s (n-m)
  \)
  for odd $r$ and $s$ such that $r+s=10$, to the authors' knowledge no formulas are known for the corresponding sums
  \(
 \sum_{m=1}^{n-1}\sigma'_{r,s}(m,n-m)
 \)
 for the same $r$ and $s$.
 \end{rmk}
 Throughout $n$ will denote a positive integer which is greater than $1$ and $p$ will denote a prime number.
 We shall sometimes write $(m,n)$ for $\gcd(m,n)$.
 Next $\phi(n)$ will
 denote the Euler totient function. For our current purposes we record the following
 well-known properties of these two functions. If $n>1$ and $s \in \mathbb{Z}\setminus\{0\}$ , then
 \begin{equation} \label{basic-phi-mu}
 \psi_s (n)= \prod_{p\mid n} (1-p^s),\
 \phi(n) = n n \psi_{-1}(n), \
 \sum_{d\mid n} \phi(d) = n,\
 \sum_{d\mid n} \mu(d) = 0.
 \end{equation}
 Note that
 \begin{equation} \label{composition-phi}
 \sum_{\substack{(x,y)\in\mathbb{N}^2 \\ x+y=n\\ \gcd(x,y)=1}} 1 = \phi(n),\ \text{for\ } n>1 \quad\text{and\ }
 \sum_{\substack{(x,y)\in\mathbb{N}_0\times \mathbb{N} \\ x+y=n\\ \gcd(x,y)=1}} 1 = \phi(n)\ \text{for\ } n\geq 1.
 \end{equation}
 Further, it is easy to verify the following formula which is crucial in this paper. If $k,n \in\mathbb{N}$ such that $n>1$, then
 \begin{equation} \label{primitive-sum}
 \sum_{\substack{1\leq t <n\\ (t,n)=1}} t^{k} = \sum_{d\mid n} \mu(d) d^{k}\sum_{j=1}^{\frac{n}{d}} j^{k}
 =
 \sum_{d\mid n} \mu(d) d^{k}\sum_{j=1}^{\frac{n}{d}-1} j^{k}
 \end{equation}
 \[
 =
 \sum_{d\mid n} \mu(d) \frac{d^{k}}{k+1} \sum_{j=0}^{k} \binom{k+1}{j} B_j \left(\frac{n}{d} \right)^{k+1-j},
 \]
 where $B_j$ denotes the $j$th Bernoulli number for which $B_1 = -1/2$ and $B_{2j+1}=0$ for $j\in \mathbb{N}$ and the first few
 terms for even $j$ are
 \[
 B_0=1,\ B_2 = 1/6, B_4 = -1/30, B_6 = 1/42, B_8 =   -1/30, B_{10}=\frac{5}{66}, B_{12}=\frac{-691}{2730}.
 \]

 \noindent
 We now state the applications of Main Theorem which we intend to prove in this work.
 \begin{theorem}\label{1,1}\cite[Theorem 1 (a)]{Elbachraoui}
 We have
 \[
 \sum_{m=1}^{n-1}\sigma'_{1,1}(m,n-m)= \frac{5n^3 -6n}{12} \psi_{-1}(n)
 + \frac{n}{12}\psi_1(n).
 \]
 \end{theorem}
 \begin{theorem} \label{1,3}
 We have
 \[
 \sum_{m=1}^{n-1}\sigma'_{1,3}(m,n-m)= \frac{7 n^5 -10 n}{10}\psi_{-1}(n) + \frac{n^3}{3}\psi_{1}(n)
 - \frac{n}{30}\psi_{3}(n).
 \]
 \end{theorem}
 \begin{theorem} \label{weight6}
 We have
 \[
 \begin{split}
 (a)\
 \sum_{m=1}^{n-1}\sigma'_{1,5}(m,n-m) &= \frac{540 n^7 -1134 n}{13608}\psi_{-1}(n) +
 \frac{n^5}{24}\psi_{1}(n) + \frac{9 n}{4536}\psi_{5}(n), \\
 (b)\
 \sum_{m=1}^{n-1}\sigma'_{3,3}(m,n-m) &= \frac{n^7}{120} \psi_{-1}(n)
  - \frac{n^3}{120} \psi_{3}(n).
 \end{split}
  \]
 \end{theorem}
 \begin{theorem} \label{weight8}
 We have
 \[
 \begin{split}
 (a)\
 \sum_{m=1}^{n-1}\sigma'_{1,7}(m,n-m) &= \frac{176 n^9 -480 n}{7680}\psi_{-1}(n) +
 \frac{n^7}{24}\psi_{1}(n) - \frac{n}{480}\psi_{7}(n), \\
 (b)\
 \sum_{m=1}^{n-1}\sigma'_{3,5}(m,n-m) &= \frac{11 n^9}{5040}\psi_{-1}(n) -
 \frac{n^5}{240}\psi_{3}(n) + \frac{n^3}{504}\psi_{5}(n).
 \end{split}
 \]
 \end{theorem}
 \begin{theorem} \label{weight12}
 We have
 \[
 \begin{split}
 (a)\
 \sum_{m=1}^{n-1}\sigma'_{1,11}(m,n-m) &= \frac{5223960 n^{13} -20638800 n}{495331200}\psi_{-1}(n) +
 \frac{n^{11}}{24}\psi_{1}(n)\\
 & - \frac{691 n}{65520} \psi_{11}. \\
 (b)\
 \sum_{m=1}^{n-1}\sigma'_{3,9}(m,n-m) &= \frac{n^{13}}{2640} \psi_{-1}(n) -
  \frac{n^9}{240} \psi_{3}(n) + \frac{n^3}{264} \psi_{9}(n).\\
 (c)\
 \sum_{m=1}^{n-1}\sigma'_{5,7}(m,n-m) &= \frac{n^{13}}{10080}\psi_{-1}(n) +
 \frac{n^7}{504}\psi_{5}(n) - \frac{n^5}{480} \psi_{7}(n).
 \end{split}
 \]
 \end{theorem}
 It worth at this point to notice the pattern in our formulas, namely:
 \begin{equation}\label{pattern}
 \sum_{m=1}^{n-1} \sigma'_{r,s}(m,n-m) = (A n^{r+s+1} + B n) \psi_{-1}(n) + C n^r \psi_{s}(n) + D n^s \psi_{r}(n).
 \end{equation}
 This motivates the following question.
 \begin{problem}
 Is it true that for all $r,s \in\mathbb{N}$ there exist $A, B, C, D \in \mathbb{Q}$ such that
 identity (\ref{pattern}) holds?
 \end{problem}
 \section{Proof of Main Theorem}
 As it was mentioned earlier, Main Theorem is an analogue of Theorem~13.1 in Huard \emph{et al.}~\cite{HOSW} 
 and the proof is essentially the same. See also Williams \cite[p. 137-140]{Williams}.
To simplify let
\begin{equation*}
g(a,b,x,y)=f(a,b,x,y)-f(x,y,a,b)
\end{equation*}
 so that
\begin{equation} \label{g-simlify}
g(a,-b,x,y)=g(-a,b,x,y)\text{ \ and \ }g(a,b,x,y)=-g(x,y,a,b).
\end{equation}
Combining (\ref{g-simlify}) with the fact that
\[
(a,b,x,y)\in\mathcal{B}'(n)\ \text{if and only if\ }
(x,y,a,b)\in\mathcal{B}'(n)\ \text{if and only if\ }
(y,x,b,a)\in\mathcal{B}'(n),
\]
we have
\begin{eqnarray*}
&&\sum_{(a,b,x,y)\in \mathcal{B}^{\prime }(n)}\bigl(%
f(a,b,x,-y)-f(a,-b,x,y)+f(a,a-b,x+y,y) \\
&&-f(a,a+b,y-x,y)+f(b-a,b,x,x+y)-f(a+b,b,x,x-y)\bigr). \\
&=&\sum_{(a,b,x,y)\in \mathcal{B}^{\prime }(n)}\bigl(%
f(a,b,x,-y)-f(x,-y,a,b)+f(a,a-b,x+y,y) \\
&&-f(x,x+y,a-b,a)+f(a-b,a,y,x+y)-f(x+y,y,a,a-b)\bigr). \\
&=&\sum_{(a,b,x,y)\in \mathcal{B}^{\prime }(n)}\bigl(%
g(a,a-b,x+y,y)+g(a-b,a,y,x+y)+g(a,b,x,-y)\bigr).
\end{eqnarray*}
On the other hand we have
\begin{eqnarray*}
&&\sum_{\substack{ 1\leq t<n  \\ (t,n)=1}}\bigl(%
f(1,0,n,t)-f(n,t,1,0)+f(0,1,t,n) \\
&&-f(t,n,0,1)+f(1,1,n-t,-t)-f(n-t,-t,1,1)\bigr).
\end{eqnarray*}%
\begin{equation*}
=\sum_{\substack{ 1\leq t<n  \\ (t,n)=1}}\bigl(%
g(1,0,n,t)+g(0,1,,t,n)+g(1,1,n-t,-t)\bigr).
\end{equation*}
Therefore identity (\ref{main-eq}) simplifies to
\begin{equation} \label{shorter}
\sum_{(a,b,x,y)\in \mathcal{B}^{\prime }(n)}\bigl(g(a,a-b,x+y,y)+g(a-b,a,y,x+y)+g(a,b,x,-y)\bigr)
\end{equation}
\[
=\sum_{\substack{ 1\leq t<n  \\ (t,n)=1}}\bigl(g(1,0,n,t)+g(0,1,t,n)+g(1,1,n-t,-t)\bigr).
\]
We will consider three cases. First note that $(a,b,x,y)\in \mathcal{B}'(n)$ and $a=b$
means that $a=b=1$. Therefore, considering the terms with $a=b$ the left hand side of~(\ref{shorter}) becomes
\begin{eqnarray*}
&&\sum_{\substack{ (a,b,x,y)\in \mathcal{B}^{\prime }(n)  \\ a=b=1}}\bigl(%
g(a,a-b,x+y,y)+g(a-b,a,y,x+y)+g(a,b,x,-y)\bigr) \\
&=&\sum_{\substack{ x+y=n  \\ (x,y)=1}}\bigl(%
g(1,0,x+y,y)+g(0,1,y,x+y)+g(1,1,x,-y)\bigr) \\
&=&\sum_{\substack{ 1\leq t<n  \\ (t,n)=1}}\bigl(%
g(1,0,n,t)+g(0,1,,t,n)+g(1,1,n-t,-t)\bigr).
\end{eqnarray*}
Secondly, we consider the terms with $a<b$. Noticing that if $a<b$ and $(a,b,x,y)\in \mathcal{B}^{\prime }(n)$,
then $a(x+y)+(b-a)y \in \mathcal{B}^{\prime }(n)$
and $(a,(b-a),(x+y),y)\in \mathcal{B}^{\prime }(n)$,
the left hand side of identity (\ref{shorter}) becomes
\begin{eqnarray*}
&&\sum_{\substack{ (a,b-a,x+y,y)\in \mathcal{B}^{\prime }(n) \\ a<b}}\bigl(%
g(a,a-b,x+y,y)+g(a-b,a,y,x+y)\bigr) \\
&&+\sum_{\substack{ (a,b,x,y)\in \mathcal{B}^{\prime }(n) \\ a<b}}g(a,b,x,-y)
\\
&=&\sum_{\substack{ (a,b,x,y)\in \mathcal{B}^{\prime }(n) \\ x>y}}\bigl(%
g(a,-b,x,y)+g(-b,a,y,x)\bigr) \\
&&+\sum_{\substack{ (a,b,x,y)\in \mathcal{B}^{\prime }(n) \\ x<y}}g(x,y,a,-b)
\\
&=&\sum_{\substack{ (a,b,x,y)\in \mathcal{B}^{\prime }(n) \\ x>y}}%
g(a,-b,x,y)+\sum_{\substack{ (a,b,x,y)\in \mathcal{B}^{\prime }(n) \\ x>y}}%
g(b,-a,y,x) \\
&&-\sum_{\substack{ (a,b,x,y)\in \mathcal{B}^{\prime }(n) \\ x<y}}g(a,-b,x,y)
\\
&=&-\sum_{\substack{ (a,b,x,y)\in \mathcal{B}^{\prime }(n) \\ x>y}}%
g(x,y,a,-b)+\sum_{\substack{ (a,b,x,y)\in \mathcal{B}^{\prime }(n) \\ x<y}}%
g(a,-b,x,y) \\
&&-\sum_{\substack{ (a,b,x,y)\in \mathcal{B}^{\prime }(n) \\ x<y}}g(a,-b,x,y)
\\
&=&-\sum_{\substack{ (a,b,x,y)\in \mathcal{B}^{\prime }(n) \\ x>y}}%
g(x,y,a,-b)
\end{eqnarray*}
Finally, we consider the terms with $a>b$. The left hand side of (\ref{shorter}) equals
\begin{eqnarray*}
&&\sum_{\substack{ (a,b,x,y)\in \mathcal{B}^{\prime }(n) \\ a>b}}\bigl(%
g(a,a-b,x+y,y)+g(a-b,a,y,x+y)+g(a,b,x,-y)\bigr) \\
&=&\sum_{\substack{ (a,b,x,y)\in \mathcal{B}^{\prime }(n) \\ a>b}}\bigl(%
g(a,a-b,x+y,y)+g(a-b,a,y,x+y)\bigr) \\
&&+\sum_{\substack{ (a,b,x,y)\in \mathcal{B}^{\prime }(n) \\ a>b}}g(a,b,x,-y).
\end{eqnarray*}
But with the help of (\ref{g-simlify}) we find
\begin{eqnarray*}
&&\sum_{\substack{ (a,b,x,y)\in \mathcal{B}^{\prime }(n) \\ a>b}}\bigl(%
g(a,a-b,x+y,y)+g(a-b,a,y,x+y)\bigr) \\
&=&\sum_{(a,b,x,x+y)\in \mathcal{B}^{\prime }(n)}\bigl(%
g(a+b,a,x+y,y)+g(a,a+b,y,x+y)\bigr)\text{ }
\end{eqnarray*}
\begin{eqnarray*}
&=&\sum_{\substack{ (a,b,x,y)\in \mathcal{B}^{\prime }(n) \\ y>x}}\bigl(%
g(a+b,a,y,y-x)+g(a,a+b,y-x,y)\bigr) \\
&=&\sum_{\substack{ (a,b,x,y)\in \mathcal{B}^{\prime }(n) \\ a>b}}\bigl(%
g(x+y,y,a,a-b)+g(y,x+y,a-b,a)\bigr) \\
&=&-\sum_{\substack{ (a,b,x,y)\in \mathcal{B}^{\prime }(n) \\ a>b}}\bigl(%
g(a,a-b,x+y,y)+g(a-b,a,y,x+y)\bigr).
\end{eqnarray*}
That is,
\begin{equation*}
\sum_{\substack{ (a,b,x,y)\in \mathcal{B}^{\prime }(n)  \\ a>b}}\bigl(%
g(a,a-b,x+y,y)+g(a-b,a,y,x+y)\bigr)=0
\end{equation*}
and thus
\begin{eqnarray*}
&&\sum_{\substack{ (a,b,x,y)\in \mathcal{B}^{\prime }(n) \\ a>b}}\bigl(%
g(a,a-b,x+y,y)+g(a-b,a,y,x+y)\bigr) \\
&&+\sum_{\substack{ (a,b,x,y)\in \mathcal{B}^{\prime }(n) \\ a>b}}g(a,b,x,-y)
\\
&=&\sum_{\substack{ (a,b,x,y)\in \mathcal{B}^{\prime }(n) \\ a>b}}%
g(a,b,x,-y)=\sum_{\substack{ (a,b,x,y)\in \mathcal{B}^{\prime }(n) \\ x>y}}%
g(x,y,a,-b)
\end{eqnarray*}
which complete the proof.
 \section{Proof of other theorems}

 \noindent
 \emph{Proof of Theorem \ref{1,1}.}\
 Let $f(a,b,x,y)=x^2$. Then clearly for all $a,b,x,y \in\mathbb{Z}$
 \[
 f(a,b,x,y)-f(x,y,a,b)= f(-a,-b,x,y)- f(x,y,-a,-b).
 \]
 With this choice the left hand side and the right hand side of equation (\ref{main-eq}) are respectively
 \[
 4 \sum_{(a,b,x,y)\in \mathcal{B}'(n)} xy\quad \text{and\quad }
 \sum_{\substack{1\leq t< n\\ (t,n)=1}} (2t^2 - 2nt +2n^2 -2).
 \]
 Then by Main Theorem and relation (\ref{pre-identity}) we obtain
 \begin{equation} \label{help 1-1-1}
 \sum_{m=1}^{n-1}\sigma'_{1,1}(m,n-m)= \sum_{(a,b,x,y)\in \mathcal{B}'(n)} xy
 \end{equation}
 \[
 = \frac{1}{2} \sum_{\substack{1\leq t <n\\ (t,n)=1}} t^2
 -\frac{n}{2} \sum_{\substack{1\leq t <n\\ (t,n)=1}} t + \frac{n^2 -1}{2} \sum_{\substack{1\leq t <n\\ (t,n)=1}} 1.
 \]
 Further by the relations in (\ref{primitive-sum}) and (\ref{basic-phi-mu}) we find
 \begin{equation} \label{help 1-1-2}
 \begin{split}
 \sum_{\substack{1\leq t <n\\ (t,n)=1}} 1 &= n \psi_{-1}(n) \\
 \sum_{\substack{1\leq t <n\\ (t,n)=1}} t &= \frac{n^2}{2} \psi_{-1}(n) \\
 \sum_{\substack{1\leq t <n\\ (t,n)=1}} t^2 &= \frac{n^3}{3} \psi_{-1}(n)
  + \frac{n}{6}\psi_{1}(n).
 \end{split}
 \end{equation}
 Now by virtue of formulas (\ref{help 1-1-1}) and (\ref{help 1-1-2}) we have
\[
\begin{split}
 \sum_{m=1}^{n-1}\sigma'_{1,1}(m,n-m) &= \frac{1}{2}
  \left( \frac{n^3}{3} \psi_{-1}(n) + \frac{n}{6} \psi_1 (n) \right) - \frac{n}{2}\left(\frac{n^2}{2}\psi_{-1}(n) \right)
  + \frac{(n^2-1)n}{2} \psi_{-1}(n) \\
 &=
 \frac{5n^3 - 6n}{12} \psi_{-1}(n) + \frac{n}{12}\psi_{1}(n).
 \end{split}
 \]
 This completes the proof.
 \\
 \\
 \noindent
 \emph{Proof of Theorem \ref{1,3}.}\
 Considering the function $f(a,b,x,y)= x^2 y^2$, the left hand side and the right hand side of equation (\ref{main-eq}) are respectively
 \[
  8 \sum_{(a,b,x,y)\in \mathcal{B}'(n)} x^3 y \quad \text{and\quad }
 \sum_{\substack{1\leq t< n\\ (t,n)=1}} (t^4 -2n t^3 + 3n^2 t^2 -1).
 \]
 Moreover it is easily checked that this function satisfies the condition of Main~Theorem and therefore with the help
 relation (\ref{pre-identity}) we obtain
 \begin{equation} \label{help 1-3-1}
 \sum_{m=1}^{n-1}\sigma'_{1,3}(m,n-m)=\sum_{(a,b,x,y)\in \mathcal{B}'(n)} x^3 y
 \end{equation}
 \[
 = \frac{1}{8} \sum_{\substack{1\leq t <n\\ (t,n)=1}} t^4
 -\frac{n}{4} \sum_{\substack{1\leq t <n\\ (t,n)=1}} t^3 + \frac{3n^2}{8} \sum_{\substack{1\leq t <n\\ (t,n)=1}} t^2
 - \frac{1}{8} \sum_{\substack{1\leq t <n\\ (t,n)=1}} 1.
 \]
 Further by the formulas in (\ref{primitive-sum}) and (\ref{basic-phi-mu}) we find
 \begin{equation} \label{help 1-3-2}
 \begin{split}
 \sum_{\substack{1\leq t <n\\ (t,n)=1}} t^3 &= \frac{n^4}{4} \psi_{-1}(n)
  + \frac{n^2}{4}\psi_{1}(n) \\
 \sum_{\substack{1\leq t <n\\ (t,n)=1}} t^4 &= \frac{n^5}{5} \psi_{-1}(n)
  + \frac{n^3}{3} \psi_{1}(n) - \frac{n}{30} \psi_{3}(n).
 \end{split}
 \end{equation}
 Combining formulas (\ref{help 1-1-2}), (\ref{help 1-3-1}), and (\ref{help 1-3-2}) we deduce the desired identity.
\\ \\
 \noindent
 \emph{Proof of Theorem \ref{weight6}.}\
 (a)\quad Application of Main~Theorem to the function $f(a,b,x,y)= x y^5 - 10 x^3 y^3$ gives
 \[
 -108 \sum_{(a,b,x,y)\in \mathcal{B}'(n)} x y^5 =
 \sum_{\substack{1\leq t< n\\ (t,n)=1}} (-9 t^6 + 30 n t^5 -30 n^2 t^4 -10 n^3 t^3 + n^5 t + 9),
 \]
 which by relation (\ref{pre-identity}) yields
 \begin{equation} \label{help 1-5-1}
 \begin{split}
 \sum_{m=1}^{n-1}\sigma'_{1,5}(m,n-m) &=\sum_{(a,b,x,y)\in \mathcal{B}'(n)} x y^5 \\
 & = \frac{9}{108} \sum_{\substack{1\leq t <n\\ (t,n)=1}} t^6
 -\frac{30 n}{108} \sum_{\substack{1\leq t <n\\ (t,n)=1}} t^5 +
 \frac{30 n^2}{108} \sum_{\substack{1\leq t <n\\ (t,n)=1}} t^4 \\
 & + \frac{10 n^3}{108}\sum_{\substack{1\leq t <n\\ (t,n)=1}} t^3
 - \frac{n^5}{108}\sum_{\substack{1\leq t <n\\ (t,n)=1}} t
 - \frac{9}{108} \sum_{\substack{1\leq t <n\\ (t,n)=1}} 1.
 \end{split}
 \end{equation}
 Further by (\ref{primitive-sum}) and (\ref{basic-phi-mu}) we get
 \begin{equation} \label{help 1-5-2}
 \begin{split}
 \sum_{\substack{1\leq t <n\\ (t,n)=1}} t^5 &= \frac{n^6}{6} \psi_{-1}(n)
  + \frac{5 n^4}{12} \psi_{1}(n) - \frac{n^2}{12} \psi_{3}(n) \\
 \sum_{\substack{1\leq t <n\\ (t,n)=1}} t^6 &= \frac{n^7}{7} \psi_{-1}(n)
  + \frac{n^5}{2} \psi_{1}(n) - \frac{n^3}{6} \psi_{3}(n)  + \frac{n}{42}\psi_{5}(n).
 \end{split}
 \end{equation}
 Combining formulas (\ref{help 1-1-2}), (\ref{help 1-3-2}), (\ref{help 1-5-1}), and (\ref{help 1-5-2}) we deduce the
 result of part (a).

 (b)\quad By Main~Theorem applied to the function $f(a,b,x,y)= x y^5 - x^3 y^3$ we find
 \[
 18 \sum_{(a,b,x,y)\in \mathcal{B}'(n)} x^3 y^3
  = \sum_{\substack{1\leq t< n\\ (t,n)=1}} (3n t^5 -3 n^2 t^4 - n^3 t^3 + n^5 t).
 \]
 Then by the relation (\ref{pre-identity}) we have
 \begin{equation} \label{help 3-3-1}
 \sum_{m=1}^{n-1}\sigma'_{3,3}(m,n-m)=\sum_{(a,b,x,y)\in \mathcal{B}'(n)} x^3 y^3
 \end{equation}
 \[
 = \frac{n}{6} \sum_{\substack{1\leq t <n\\ (t,n)=1}} t^5
 -\frac{n^2}{6} \sum_{\substack{1\leq t <n\\ (t,n)=1}} t^4
 - \frac{n^3}{18}\sum_{\substack{1\leq t <n\\ (t,n)=1}} t^3
 + \frac{n^5}{18}\sum_{\substack{1\leq t <n\\ (t,n)=1}} t,
 \]
 which by (\ref{help 1-1-2}), (\ref{help 1-3-2}), and (\ref{help 1-5-2}) implies the result.
\\ \\
 \noindent
 \emph{Proof of Theorem \ref{weight8}.}\
 (a)\quad Application of Main~Theorem to the function $f(a,b,x,y)= -22 x^7 y +112 x^5 y^3$ gives
 \[
 1440 \sum_{(a,b,x,y)\in \mathcal{B}'(n)} x y^7
 \]
 \[
 = \sum_{\substack{1\leq t< n\\ (t,n)=1}} (90 t^8 - 428 n t^7 +658 n^2 t^6 -238 n^3 t^5 -210 n^4 t^4 +
  462 n^5 t^3 -154 n^6 t^2 - 90),
 \]
 which by (\ref{pre-identity}) yields
 \begin{equation} \label{help 1-7-1}
 \begin{split}
 \sum_{m=1}^{n-1}\sigma'_{1,7}(m,n-m) &= \sum_{(a,b,x,y)\in \mathcal{B}'(n)} x y^7 \\
 & = \frac{90}{1440} \sum_{\substack{1\leq t <n\\ (t,n)=1}} t^8
 -\frac{428 n}{1440} \sum_{\substack{1\leq t <n\\ (t,n)=1}} t^7
 + \frac{658 n^2}{1440} \sum_{\substack{1\leq t <n\\ (t,n)=1}} t^6 \\
 & - \frac{238 n^3}{1440}\sum_{\substack{1\leq t <n\\ (t,n)=1}} t^5
  - \frac{210 n^4}{1440}\sum_{\substack{1\leq t <n\\ (t,n)=1}} t^4
 + \frac{462 n^5}{1440}\sum_{\substack{1\leq t <n\\ (t,n)=1}} t^3 \\
 & - \frac{154 n^6}{1440}\sum_{\substack{1\leq t <n\\ (t,n)=1}} t^2
 - \frac{90}{1440} \sum_{\substack{1\leq t <n\\ (t,n)=1}} 1.
 \end{split}
 \end{equation}
 Further by the relations in (\ref{primitive-sum}) and (\ref{basic-phi-mu}) we find
 \begin{equation} \label{help 1-7-2}
 \begin{split}
 \sum_{\substack{1\leq t <n\\ (t,n)=1}} t^7 &= \frac{n^8}{8} \psi_{-1}(n)
  + \frac{7 n^6}{12} \psi_{1}(n) - \frac{7 n^4}{24} \psi_{3}(n)
  + \frac{7 n^2}{84}\psi_{5}(n) \\
 \sum_{\substack{1\leq t <n\\ (t,n)=1}} t^8 &= \frac{n^9}{9} \psi_{-1}(n)
  + \frac{2 n^7}{3} \psi_{1}(n) - \frac{7 n^5}{15} \psi_{3}(n)
  + \frac{2 n^3}{9}\psi_{5}(n) - \frac{n}{30} \psi_{7}(n). \\
  \end{split}
 \end{equation}
 Now use (\ref{help 1-1-2}), (\ref{help 1-3-2}), (\ref{help 1-5-2}),  (\ref{help 1-7-1}), and (\ref{help 1-7-2}) to
 conclude the result.

 (b)\quad By Main~Theorem applied to
 the function $f(a,b,x,y)= x^7 y - x^5 y^3$ we have
 \[
 90 \sum_{(a,b,x,y)\in \mathcal{B}'(n)} x^3 y^5
 =\sum_{\substack{1\leq t< n\\ (t,n)=1}} (-n t^7 +11 n^2 t^6-26 n^3 t^5 +30n^4 t^4 -21 n^5 t^3 +7 n^6 t^2).
 \]
 This fact combined with the relations
 (\ref{help 1-1-2}), (\ref{help 1-3-2}), (\ref{help 1-5-2}),  and (\ref{help 1-7-2}) gives
 the result.
\\ \\
 \noindent
 \emph{Proof of Theorem \ref{weight12}.}\
 (a)\quad Considering the function
 $f(a,b,x,y)= 271 x^{11} y -1540 x^9 y^3 + 1584 x^7 y^5$ in Main~Theorem we have
 \[
 \begin{split}
 7560 \sum_{(a,b,x,y)\in \mathcal{B}'(n)} x y^{11} &=
 \sum_{\substack{1\leq t< n\\ (t,n)=1}} (315 t^{12}+ 62 n t^{11}
  - 7271 n^2 t^{10} +27665 n^3 t^9\\
  & -49170 n^4 t^8 + 37158 n^5 t^7 + 6930 n^6 t^6 -33990 n^7 t^5\\
  & + 30855 n^8 t^4 - 14905 n^9 t^3 + 2981 n^{10} t^2 -315).
 \end{split}
 \]
  Combining this fact with the identities (\ref{help 1-1-2}), (\ref{help 1-3-2}), (\ref{help 1-5-2}), (\ref{help 1-7-2})
  along with
 \begin{equation} \label{help 12}
 \begin{split}
 \sum_{\substack{1\leq t <n\\ (t,n)=1}} t^9 &= \frac{n^{10}}{10} \psi_{-1}(n)
  + \frac{3 n^8}{4} \psi_{1}(n) - \frac{7 n^6}{10} \psi_{3}(n) + \frac{n^4}{2}\psi_{5}(n) \\
 & - \frac{3 n^2}{20} \psi_{7}(n) \\
 \sum_{\substack{1\leq t <n\\ (t,n)=1}} t^{10} &= \frac{n^{11}}{11} \psi_{-1}(n)
  + \frac{5 n^9}{6} \psi_{1}(n) - n^7 \psi_{3}(n) + n^5\psi_{5}(n) \\
  &- \frac{n^3}{2} \psi_{7}(n) + \frac{5n}{66} \psi_{9}(n) \\
 \sum_{\substack{1\leq t <n\\ (t,n)=1}} t^{11} &= \frac{n^{12}}{12} \psi_{-1}(n)
  + \frac{11 n^{10}}{12} \psi_{1}(n) - \frac{11 n^8}{8} \psi_{3}(n) + \frac{11 n^6}{6} \psi_{5}(n) \\
  & - \frac{11 n^4}{8} \psi_{7}(n) + \frac{5 n^2}{12}\psi_{9}(n) \\
 \sum_{\substack{1\leq t <n\\ (t,n)=1}} t^{12} &= \frac{n^{13}}{13} \psi_{-1}(n)
   + n^{11} \psi_{1}(n) - \frac{11 n^9}{6} \psi_{3}(n) + \frac{22 n^7}{7}\psi_{5}(n) \\
   & - \frac{33 n^5}{10} \psi_{7}(n)
   + \frac{5 n^3}{3}\psi_{9}(n) - \frac{691 n}{2730} \psi_{11}(n)
 \end{split}
 \end{equation}
 yields the result.

 (b)\quad As to this part, using
 the function $f(a,b,x,y)= -2 x^{11} y +11 x^9 y^3 -9 x^7 y^5$ in Main~Theorem, we find
 \[
 \begin{split}
 180 \sum_{(a,b,x,y)\in \mathcal{B}'(n)} x^3 y^{9} &=
 \sum_{\substack{1\leq t< n\\ (t,n)=1}} (-16n t^{11} +97 n^2 t^{10} - 268 n^3 t^{9} +411 n^4 t^8 -282 n^5 t^7\\
 & - 63 n^6 t^6 + 264 n^7 t^5 -231 n^8 t^4+ 110 n^9 t^3 - 22 n^{10} t^2 ).
 \end{split}
 \]
 Now use the relations (\ref{pre-identity}), (\ref{help 1-1-2}), (\ref{help 1-3-2}), (\ref{help 1-5-2}),
 (\ref{help 1-7-2}), and (\ref{help 12}) to deduce the desired formula.

 (c)\quad Use the function $f(a,b,x,y)= 8 x^{11} y -35 x^9 y^3 +27 x^7 y^5$ in Main~Theorem to obtain
 \[
 \begin{split}
 2520 \sum_{(a,b,x,y)\in \mathcal{B}'(n)} x^5 y^{7} &=
 \sum_{\substack{1\leq t< n\\ (t,n)=1}} (46n t^{11} -253 n^2 t^{10} + 640 n^3 t^{9} -825 n^4 t^8 +174 n^5 t^7\\
 & + 945 n^6 t^6 - 1380 n^7 t^5 +1005 n^8 t^4- 440 n^9 t^3 + 88 n^{10} t^2 )
 \end{split}
 \]
 and argue as before to conclude the result.
 \section{Concluding remarks on integer representations}
 In this section we will show that each of the sums $\sum_{m=1}^{n-1} \sigma_r (m) \sigma_s (n-m)$ and
 $\sum_{m=1}^{n-1} \sigma'_{r,s} (m,n-m)$ count certain representations of positive integers involving radicals.
 \begin{definition}
 For $1<n \in\mathbb{N}$ and $r,s \in\mathbb{N}$ let the functions $L_{r,s}(n)$, $M_{r,s}(n)$,
 $L'_{r,s}(n)$ and $M'_{r,s}(n)$ be defined as follows
 \[
 L_{r,s}(n) = \# \{ (a,b,c,d,x,y)\in \mathbb{N}_{0}^2 \times \mathbb{N}^4:\
  ( \sqrt[r]{a+c}, \sqrt[s]{b+d}, x, y) \in \mathcal{B}(n) \}
 \]
 \[
 \begin{split}
 M_{r,s}(n) = &\# \{ (a,b,c,d,x,y,k,l)\in \mathbb{N}_0^{2}\times\mathbb{N}^6:\
 ( \sqrt[r]{k(a+c)}, \sqrt[s]{l(b+d)},x,y)\in \mathcal{B}(n)\\
 &\quad\quad \text{and\ } (a,c)=(b,d)=1 \}.
 \end{split}
 \]
 \[
 L'_{r,s}(n) = \# \{ (a,b,c,d,x,y)\in \mathbb{N}_{0}^2 \times \mathbb{N}^4:\
  ( \sqrt[r]{a+c}, \sqrt[s]{b+d}, x, y) \in \mathcal{B}'(n) \}
 \]
 \[
 \begin{split}
 M'_{r,s}(n) = &\# \{ (a,b,c,d,x,y,k,l)\in \mathbb{N}_0^{2}\times\mathbb{N}^6:\
 ( \sqrt[r]{k(a+c)}, \sqrt[s]{l(b+d)},x,y)\in \mathcal{B}'(n) \\
 & \text{and\ } (a,c)=(b,d)=1 \}.
 \end{split}
 \]
 \end{definition}
 \begin{theorem} \label{L-M}
 If $1<n \in\mathbb{N}$ and $r,s \in\mathbb{N}$, then
 \[
 (a)\ L_{r,s}(n) = M_{r,s}(n) = \sum_{m=1}^{n-1} \sigma_r (m) \sigma_s (n-m) .
 \]
 \[
 (b)\ L'_{r,s}(n) = M'_{r,s}(n) = \sum_{m=1}^{n-1} \sigma_{r,s} (m,n-m).
 \]
 \end{theorem}
 \begin{proof}
 (a)\quad First note that
 \begin{equation} \label{concluding help}
 \sum_{m=1}^{n-1} \sigma_r (m) \sigma_s (n-m) = \sum_{(a,b,x,y)\in \mathcal{B}'(n)} a^r b^s.
 \end{equation}
 Further we have
 \[
 \begin{split}
 L_{r,s}(n)
 &=
 \sum_{\substack{(a,b,c,d,x,y)\in\mathbb{N}_{0}^2 \times \mathbb{N}^4 \\ ( \sqrt[r]{a+c},\sqrt[s]{b+d},x,y) \in \mathcal{B}(n)}} 1 \\
 &=
 \sum_{(u,v,x,y) \in \mathcal{B}(n)} \Bigl(\sum_{\substack{(a,c)\in \mathbb{N}_0 \times \mathbb{N}\\ a+c=u^r}} 1 \Bigr)
  \Bigl(\sum_{\substack{(b,d)\in \mathbb{N}_0 \times \mathbb{N}\\ b+d=v^s}} 1 \Bigr) 1 \\
 &=
 \sum_{(u,v,x,y) \in \mathcal{B}(n)} u^r v^s
 \end{split}
 \]
 and
 \[
 \begin{split}
 M_{r,s}(n)
 &=
 \sum_{\substack{(a,b,c,d,x,y,k,l)\in\mathbb{N}_0^{2}\times\mathbb{N}^6 \\ (\sqrt[r]{k(a+c)},\sqrt[s]{l(b+d)},x,y)\in\mathcal{B}(n)\\(a,c)=(b,d)=1}} 1 \\
 &=
 \sum_{(u,v,x,y)\in\mathcal{B}(n)} \Bigl(\sum_{e\mid u^r}\sum_{\substack{(a,c)\in\mathbb{N}_0\times\mathbb{N}\\ a+c =e\\ (a,c)=1}} 1\Bigr)
  \Bigl(\sum_{f\mid v^s}\sum_{\substack{(b,d)\in\mathbb{N}_0\times\mathbb{N}\\ b+d =f\\ (b,d)=1}} 1\Bigr) \\
 &=
 \sum_{(u,v,x,y)\in\mathcal{B}(n)} \Bigl(\sum_{e\mid u^r} \phi(e) \Bigr) \Bigl(\sum_{f\mid v^s} \phi(f) \Bigr) \\
 &=
 \sum_{(u,v,x,y)\in\mathcal{B}(n)} u^r v^s,
 \end{split}
 \]
 which with the help of (\ref{concluding help}) completes the proof of part (a).

 (b)\quad This part follows similarly with an application of relation (\ref{pre-identity}) instead of relation
 (\ref{concluding help}).
 \end{proof}
 %
 
%

\begin{thebibliography}{9}
\bibitem{Glaisher1}
 J. W. L. Glaisher,
 \emph{On the square of the series in which the coefficients are the sums of the divisors of the exponents},
 Mess. Math. {\bf 14} (1884), 156-163.
\bibitem{Glaisher3}
 J. W. L. Glaisher,
 \emph{Expressions for the first five powers of the series in which the coefficients are the sums of the divisors of the exponents},
 Mess. Math. {\bf 15} (1885), 33-36.
\end{thebibliography}
\end{document}